\documentclass[12pt]{amsart}
\usepackage{amscd}
\usepackage{amsfonts}
\usepackage[dvipsnames]{xcolor}
\usepackage{amsmath,amssymb,amsthm} 
\usepackage[compatible]{algpseudocode}
\usepackage{algorithm}

\usepackage{url}
\usepackage{cite}

\usepackage{tikz}
\usepackage[all]{xy}
\usepackage{enumitem} 

\newtheorem{thm}{Theorem}[section]
\newtheorem{cor}[thm]{Corollary}

\theoremstyle{definition}
\newtheorem{ex}[thm]{Example}
\newtheorem{lem}[thm]{Lemma}

\theoremstyle{definition}
\newtheorem{defn}[thm]{Definition}

\numberwithin{equation}{section}

\def\P{\mathbb{CP}} 
\def\N{\mathbb{N}} 
\def\Z{\mathbb{Z}} 
\def\R{\mathbb{R}} 
\def\C{\mathbb{C}} 
\def\O{\mathcal{O}} 

\def\F{\mathcal{F}}
\newcommand{\XX}{{\mathcal X}}
\newcommand{\pa}[1]{\frac{\partial}{\partial #1}}

\definecolor{miazul}{rgb}{0.60,0.75,0.90}
\definecolor{verdeazul}{rgb}{0.30,0.80,0.30}

\usepackage{pgf,tikz}
\usetikzlibrary{arrows}

\definecolor{wwffww}{rgb}{0.4,1,0.4}
\definecolor{ffttzz}{rgb}{1,0.2,0.6}
\definecolor{ffttqq}{rgb}{1,0.2,0}
\definecolor{qqzzzz}{rgb}{0,0.6,0.6}
\definecolor{qqqqff}{rgb}{0,0,1}
\definecolor{qqccww}{rgb}{0,0.8,0.4}
\definecolor{uququq}{rgb}{0.25,0.25,0.25}
\definecolor{ffqqww}{rgb}{1,0,0.4}

\makeatother

\begin{document}     

\title[Projective foliations]{Projective foliations with a unique singular point and maximum algebraic  multiplicity}
 


\author{Jorge Mozo-Fern\'{a}ndez}
\address[Jorge Mozo Fern\'{a}ndez]{Dpto. \'{A}lgebra, An\'{a}lisis Matem\'{a}tico, Geometr\'{\i}a y Topolog\'{\i}a and Instituto de Investigación en Matemáticas - UVa (IMUVA)\\
Facultad de Ciencias, Universidad de Valladolid \\
Campus Miguel Delibes\\
Paseo de Bel\'{e}n, 7\\
47011 Valladolid - Spain}
\email{jorge.mozo@uva.es}

\thanks{ Both authors supported by Agencia Estatal de Investigación, of  Ministerio de Ciencia, Innovación y Universidades, under Project  \textit{Análisis Asintótico, Álgebra y Geometría en Sistemas Dinámicos}, Ref. PID2022-139631NB-I00, and GIR ECSING of the University of Valladolid. Financial support of the Department of Education of the Junta de Castilla y León and FEDER Funds is gratefully acknowledged (Reference: CLU-2025-1-02- IMUVA)}

\author{P. Rubí Pantaleón-Mondragón}
\address[P. Rubí Pantaleón-Mondragón]{Facultad de Ciencias de la UNAM\\
 C.U. Coayacán, C.P. 04510. Ciudad de México, México \\}
\email{pantaleon@ciencias.unam.mx}


\begin{abstract}
The study of codimension-one holomorphic foliations on the projective plane $\P^{2}$ with a unique singular point has attracted increasing interest because of its connections with several important problems. In this paper, we focus on holomorphic foliations of degree $d$ having a singular point of maximal multiplicity $d$. We describe explicitly the foliations in this class and characterize those having a unique singular point. We also state some results concerning Hilbert-Mumford stability for these foliations.
\end{abstract}


\maketitle

\section{Introduction} 
A codimension-one holomorphic foliation $\F$ of degree $d$ on the complex projective plane can be described in two equivalent ways:
\begin{enumerate}
\item By a homogeneous vector field
$$
\XX= A(x,y,z)\pa{x}+ B(x,y,x)\pa{y} + C(x,y,z)\pa{z},
$$
where $A$, $B$, $C\in \C [x,y,z]$ are homogeneous polynomials of the same degree $d$. This vector field describing the foliation is not unique: if $G(x,y,z)$ is a homogeneous polynomial of degree $d-1$ and $\mathcal{R}=x\pa{z}+y\pa{y}+z\pa{z}$ is the radial vector field, both $\XX$ and $\XX + G\cdot \mathcal{R}$ describe the same foliation and conversely. As a matter of notation, when describing a foliation with a vector field $\XX$, we will sometimes write concisely
$$
\XX=\begin{pmatrix} A \\ B \\ C \end{pmatrix}.$$

\item By a homogeneous 1-form $\omega=P(x,y,z)dx+ Q(x,y,z) dy+ R(x,y,x) dz$, with $P$, $Q$ and $R$ homogeneous of degree $d+1$ satisfying Euler's condition $xP+yQ+zR=0$.
\end{enumerate}

Both expressions for a holomorphic foliation are related as follows. If $\F$ is \mbox{described} by a vector field $\XX$ as above, the 1-form defining the foliation can be written as
$$
\begin{vmatrix}
dx & dy & dz \\ x & y & z \\ A & B & C
\end{vmatrix}= (yC-zB)dx+(zA-xC)dy+(xB-yA)dz.
$$
We denote by $\F(d;2)$ the space of degree $d$  foliations on $\P^{2}$, which turns out to be a projective space of dimension $d^{2}+4d+2$.
We shall assume that $\F$ has isolated singularities, which means that the coefficients have no common factors. It is known that such a foliation is determined by its singular subscheme $Sing(\F)$ if $d>1$ \cite{GM-K,CO1992,CO}. 
Without loss of generality, we can assume that $p=[1:0:0]$ is a singular point of $\F$, then we can consider a local representation of $\F$ over an affine chart of $\P^{2}$ having $p$ as the origin. This local representation is given by a holomorphic $1$-form  \[\omega_{\F}=f(y,z)dy+g(y,z)dz,\]
where $f(y,z)$, $g(y,z)$ are polynomials. For convenience, sometimes we will denote the foliation $\F$ by $\omega$ instead of $\F$ or $\omega_{\F}$, when there is no risk of confusion.

A projective algebraic curve defined by a reduced homogeneous polynomial $F$ is said to be invariant by the foliation $\F$ if the vector field defining the foliation is tangent to $F$ in the regular points. When $\F$ is defined by a vector field $\XX$, this means that $\XX (F)$  is a multiple of $F$. In the language of $1-$forms, $\omega \wedge dF=F\cdot\eta$, where $\eta$ is a $2-$form.

The study of global properties of foliations makes use of several numbers associated to the singularities, which turn out to be invariants under (local) analytic diffeomorphisms. The most useful invariants we shall make use in this work are:

\begin{itemize}
    \item The {\bf Milnor number} of $\omega$ in $p$, defined as \[\mu_{p}(\omega)=\dim_{\C}\frac{\O_{\C^{2},(0,0)}}{\langle f,g\rangle\cdot\O_{\C^{2},(0,0)}},\] 
    \end{itemize}
    where $\O_{\C^{2},(0,0)}$ denotes the local ring of regular functions at $(0,0)$.
    
    Let us observe that  Milnor number is nothing but the intersection index of $f$ and $g$ at $(0,0)$, which we will denote in the sequel by $I_{0}(f,g)$.
    \begin{itemize}
    \item If we consider the decomposition of $f$ and $g$ into homogeneous components, namely 
    \begin{align*}
        f=& f_{m} +f_{m+1}+\cdots,\\
        g=& g_{s}+g_{s+1}+\cdots
    \end{align*}
    then, the {\bf algebraic multiplicity} of $\omega$ at $p$ is defined as \[m_{p}(\omega):=\min\{m,s\}.\]
\end{itemize}

In \cite{J06}, J.-P. Jouanolou shows that
\[\sum\limits_{p\in Sing(\omega)}\mu_{p}(\omega)=d^{2}+d+1.\] 

This means in particular that there always exist singularities of a projective foliation, and that the number of such singularities is bounded above by $d^2+d+1$. We denote by $\mathcal{A}_d$ the set of degree $d$ foliations having only one singular point, and $\mathcal{A}_{d,m}$ the subset of $\mathcal{A}_d$ consisting of foliations such that the singular point has algebraic multiplicity $m$. More precisely, 
\[\mathcal{A}_{d}~:=\{\F\in \F(d;2)~:~Sing(\F)=\{p\},~\mu_{p}(\F)=d^2+d+1 \},\]
\[\mathcal{A}_{d,m}~:= \{ \F\in \mathcal{A}_{d}~:~ m_{p}(\F)=m\}.\]

In \cite{AP19}, the authors proved that $\dim \mathcal{A}_{d}\geq 3d+2$, for $d>1$. The case that has received the most attention is of $\mathcal{A}_{d,1}$; in this case, the singular point may be either a nilpotent singularity or a saddle-node, and such foliations do not have  invariant algebraic curves, which is an important property for the study of foliations \cite{AP19,L22}. Note that these examples are among the few known explicit examples of foliations without invariant algebraic curves. However, there are further questions regarding these  cases, such as the number of foliations, up to change of coordinates, that is, the study of strata with Hilbert-Mumford stability \cite{CDGM,AlcantaraGrado2, A22}. 

Another reason which motivates the study of the set $\mathcal{A}_{d}$ stems from a recent work by D. Cerveau and J. Déserti \cite{cerveaudeserti} which guarantees that under a Cremona transformation, every foliation of degree $d$ on $\P^{2}$ can be mapped to one with a single singularity, so it is birational to it. Unfortunately, the degree of the new foliation may be very high, see  \cite{alcantara2024remarks} for a concrete example.

In this paper, we will focus in the opposite case to the one stated before, i.e. when the algebraic multiplicity of the singular point is maximal: in other words we will be interested in foliations belonging to the set $\mathcal{A}_{d,d}$.

The structure of the paper is as follows: in Subsection \ref{Sub:una singularidad} we characterize  foliations of degree $d$ with a singularity of algebraic multiplicity $d$ and Milnor number $d^{2}+d+1$. Due to Jouanolou's result, these foliations have only one singular point. We will prove that they have at most $d$ invariant lines. In Subsection \ref{Sub:mas puntos}, we analyze the behavior of foliations with a singularity of Milnor number at most $d^{2}+d$ and algebraic multiplicity $d$. We will present an example involving two singular points, a  result previously studied in the paper \cite{R26}. In this  example, we apply a suitable Cremona transformation to analyze the degree obtaining elements in $\mathcal{A}_{s,\frac{s+1}{2}}$ for odd $s$.  We conclude in Section \ref{stability} with some issues related to Hilbert-Mumford stability. 

Notation: in this text, given a polynomial $P$ and a monomial $\mathbf{m}$ we will denote by $C(\mathbf{m},P)$  the coefficient of the monomial  $\mathbf{m}$ in the polynomial $P$.

\section{Results: $m_{p}(\omega)=d$}

\subsection{Foliation with a unique singular point}\label{Sub:una singularidad}

In this section, we consider foliations $\F$ of degree $d$ on $\P^{2}$ having a singular point $p$ of algebraic multiplicity $d$. Without loss of generality, we may assume that $p=[1:0:0]$. We also assume that the foliation is non-dicritical, meaning that only finitely many invariant analytic curves (separatrices) pass through each singular point. In addition, we may assume that the line at infinity $x=0$ is not invariant under $\F$. Under these conditions, any such foliation can be represented locally around $p$ by the $1$-form
\begin{equation}\label{Fol:d:d}
\omega=P_{d}dy+Q_{d}dz+R_{d}(ydz-zdy),
\end{equation}
where $P_{d},Q_{d},R_{d}\in\C[y,z]_{d}$ are homogeneous polynomials of degree $d$.

Let $S(y,z)=yP_{d}+zQ_{d}$ be the equation of the tangent cone at $p$. This is a homogeneous polynomial of degree $d+1$, and it is not identically zero because $\F$ is non-dicritical. Let $M=by-az$ be a linear form, with $a,b\in\C$. A straightforward computation shows that $M$ defines an invariant line for $\omega$ if and only if
\begin{align}
S(a,b)=aP_{d}(a,b)+bQ_{d}(a,b)=0,
\end{align}
that is, if and only if $M$ is a factor of $S(y,z)$. In Corollary \ref{atmostdlines}, we will see that there are at most $d$ distinct invariant lines when the foliation $\omega$ has a unique singular point.

Let us write $S(y,z)$ as a product of $d+1$ (not necessarily distinct) linear factors $M_{k}\in \C[y,z]$. Thus,
\begin{align}\label{Eq:S}
    S(y,z)=\prod_{k=0}^{r}M_{k}^{m_{k}},
\end{align}
where $m_{k}\in\N$, $\sum_{k=0}^{r}m_{k}=d+1$, and $M_{i}\neq M_{j}$ for $i\neq j$.

\begin{thm}\label{Teo:d:d}
With the notation in \eqref{Fol:d:d}, a foliation defined by a $1$-form $\omega$ belongs to $\mathcal{A}_{d,d}$ if and only if all the factors $M_{k}$ of $S(y,z)$ divide both $P_{d}$ and $Q_{d}$.
\end{thm}

\begin{proof}
Let
\[
\omega=P_{d}dy+Q_{d}dz+R_{d}(ydz-zdy),
\]
and assume that $p=[1:0:0]\in Sing(\omega)$ with $m_{p}(\omega)=d$. Without loss of generality, after a projective change of coordinates, we may assume that $M_k\neq y,z$.

Consider the intersection index
\[
\tilde{\mu}=I_{0}(y(P_{d}-zR_{d}),z(Q_{d}+yR_{d})).
\]
Then
\[
\tilde{\mu}= I_{0}(S(y,z), z(Q_{d}+yR_{d}))
= \sum_{k=0}^{r}m_{k}I_{0}(M_{k},z(Q_{d}+yR_{d})).
\]
By elementary properties of the intersection index,
\begin{equation*}
I_{0}(M_{k},z(Q_{d}+yR_{d})) =
\left\{
\begin{array}{ll}
d+1 & \mbox{if } M_{k}\nmid Q_{d},\\
d+2 & \mbox{if } M_{k}\mid Q_{d}.
\end{array}
\right.
\end{equation*}

Observe that, since the singularities are isolated, the condition $M_k\mid Q_d$ implies that $M_{k}\nmid R_{d}$.

Consequently,
\[
\tilde{\mu}\leq \sum_{k=0}^{r}m_{k}(d+2)=(d+2)(d+1).
\]
On the other hand,
\begin{align*}
\tilde{\mu}
&= I_{0}(y,z)+I_{0}(y,Q_{d})+I_{0}(P_{d},z)+\underbrace{I_{0}(P_{d}-zR_{d},Q_{d}+yR_{d})}_{\mu} \\
&= 1+2d+\mu,
\end{align*}
so
\[
\mu=\tilde{\mu}-2d-1\leq d^{2}+d+1.
\]
We conclude that $\mu=d^2+d+1$ if and only if, for every $k\in \{0,1,\ldots,r\}$, the factor $M_k$ divides both $P_d$ and $Q_d$.

A similar computation can be carried out if, for instance, $M_0=y$. In this case, we may write $Q_d=yQ_{d-1}$, where $Q_{d-1}$ is homogeneous of degree $d-1$. Then
\begin{align*}
\mu
&= I_{0}(P_d-zR_d, Q_d+yR_d) \\
&= I_{0}(P_d-zR_d, y)+ I_{0}(P_d+zQ_{d-1}, Q_{d-1}+R_d) \\
&= I_{0}(P_d-zR_d, y) + (m_0-1)I_{0}(y, Q_{d-1}+R_d) + \sum_{k=1}^{r} I_{0}(M_k, Q_{d-1}+R_d).
\end{align*}
The first term takes the value $d$ or $d+1$, depending on whether $y\nmid P_d$ or $y\mid P_d$. Similarly, the other terms take the value $d-1$ or $d$. Therefore,
\[
\mu\leq d+1+ \left( \sum_{k=0}^{r} m_k -1 \right) d = d^2+d+1.
\]
Equality holds if and only if $y$ and the remaining linear factors $M_k$ divide both $P_d$ and $Q_d$.
\end{proof}

\begin{cor}\label{atmostdlines}
A foliation $\omega\in \mathcal{A}_{d,d}$ has at most $d$ invariant lines. In fact, assume that $S(y,z)$ has $d$ distinct factors $\{ M_k\}_{k=0}^{d-1}$. Then $\omega \in \mathcal{A}_{d,d}$ if and only if
\[
P_{d}=a\prod_{k=0}^{d-1}M_{k},
\qquad
Q_{d}=b\prod_{k=0}^{d-1}M_{k},
\]
with $a,b\in\C$, $(a,b)\neq (0,0)$, and $ay+bz=M_{k}$ for some $k$ such that $m_{k}=2$.
\end{cor}

\begin{proof}
Indeed, by Theorem \ref{Teo:d:d}, if $\omega\in \mathcal{A}_{d,d}$, then every factor of $S(y,z)$ must divide both $P_d(y,z)$ and $Q_d(y,z)$, which proves the first assertion. If $S$ has $d$ distinct factors, then these determine $P_d$ and $Q_d$ up to nonzero constants, which proves the second statement.
\end{proof}

\begin{ex}
Consider
\[
\omega=P_{3}dy+Q_{3}dz+R_{3}(ydz-zdy),
\]
where
\begin{align*}
P_{3} &= (z-y)(z-2y)(z+y),\\
Q_{3} &= (z-y)(z-2y)(z-3y),\\
R_{3} &= y^{3}.
\end{align*}
Then
\[
S(y,z)=yP_{3}+zQ_{3}=(z-y)^{3}(z-2y).
\]
Thus, using the previous notation, $\tilde{\mu}=20$ and
\[
\mu=\tilde{\mu}-(1+2\cdot 3)=13.
\]
\end{ex}

\subsection{Foliation with more singular points}\label{Sub:mas puntos}

In the sequel, we study the nature of the other singular points of the foliation, provided that they exist; that is, when some factor $M_k$ of the tangent cone does not divide either $P_d$ or $Q_d$.

\begin{lem}\label{Lemma:Mk-noFactor}
Let $S(y,z)=\prod_{k=0}^{r}M_{k}^{m_{k}}$ as in \eqref{Eq:S}. If $M_{k}$ is not a factor of either $P_{d}$ or $Q_{d}$ in the foliation $\omega$ given by \eqref{Fol:d:d}, then there exists a unique singular point of $\omega$ on the line $M_{k}$ different from $p=[1:0:0]$, with algebraic multiplicity $1$ and Milnor number $m_{k}$.
\end{lem}

\begin{proof}
After a change of coordinates, we may assume that $M_k=y$ and that $y$ is not a factor of $P_d$. Since
\[
P_d(0,z)-zR_d(0,z)=z^d\bigl(C(z^d,P_d)-zC(z^d,R_d)\bigr),
\]
we distinguish two cases.

If $C(z^d,R_d)\neq 0$, let
\[
z_k=\frac{C(z^d,P_d)}{C(z^d,R_d)}\neq 0.
\]
Then the point $q_k=[1:0:z_k]$ is singular.

If $C(z^d,R_d)=0$, we must consider the point at infinity $q_k=[0:0:1]$. In the chart $z=1$, the foliation is defined by the $1$-form
\[
-S(y,1)\,dx+(xP_{d}(y,1)-R_{d}(y,1))\,dy.
\]
Since $R_d(0,1)=C(z^d,R_d)=0$, the point $q_k$ is singular.

The Milnor number can be computed under these assumptions, after possibly applying a projective transformation. Since
\[
S(y,1)= y^{m_k} U(y),
\qquad\text{with } U(0)\neq 0,
\]
we obtain
\[
I_0(S(y,1), xP_{d}(y,1)-R_{d}(y,1))
= m_k\, I_0(y, xP_{d}(0,1)-R_{d}(0,1))
= m_k.
\]
\end{proof}

As a consequence, if a degree $d$ foliation defined by a $1$-form $\omega$ has a singular point $p$ of multiplicity $d$, then
\[
\mu_{p}(\omega)= d^2+\sum_{M_k\mid (P_d \text{ and } Q_d)} m_k.
\]

A particular situation may occur. Assume that $r=d$, and consequently $m_k=1$ for every $k$. If any linear factor $M_k$ fails to divide both $P_d$ and $Q_d$, then there are $d+1$ singular points of multiplicity $1$. If all of them lie on a line $L$, then necessarily this line is invariant under $\F$. Even when $r<d$, it may happen that all these singular points lie on a line, and that this line is invariant. Observe that if such an invariant line $L$ exists, then, since it intersects every $M_k$, there must be a singular point of $\F$ on each $M_k$.

In this case, after a projective transformation, we may assume that $L$ is the line at infinity $x=0$. The foliation is then described by a homogeneous $1$-form
\[
\omega= P_d(y,z)\,dy+Q_d(y,z)\,dz,
\]
with
\[
S(y,z)=yP_d(y,z)+zQ_d(y,z)\neq 0.
\]
This foliation has $S(y,z)$ as an integrating factor, and therefore it admits a multivalued first integral. A detailed study of affine foliations defined by homogeneous polynomials can be found in \cite{cerveaumattei}. See also \cite{linsscardua} for a general description.

\section{Remarks on stability} \label{stability}

In this section, we state some partial results concerning the Hilbert--Mumford stability of the elements of $\mathcal{A}_{d,m}$. We begin by recalling the main definitions and some standard results, following \cite{newstead51lectures}.

Let $V\subset \P^{n}$ be a nonsingular complex projective variety, and let $G$ be a reductive group acting linearly on $V$.

Let $\lambda:\C^{*}\rightarrow G$ be a one-parameter subgroup ($1$-PS) of $G$. There is an induced morphism, which we denote again by $\lambda$, given by
\[
\begin{array}{rcl}
  \lambda:\C^{*} &\longrightarrow & GL_{n+1}(\C) \\
    t &\longmapsto & \begin{array}{ccc}
      \lambda(t):  \C^{n+1} &\longrightarrow & \C^{n+1}  \\
         \nu & \longmapsto & \lambda(t)\nu.
    \end{array}
\end{array}
\]

This morphism is a diagonal representation of the torus $\C^{*}$, and it always admits a diagonal description. More precisely, there exists a basis $\{\nu_{0},\ldots,\nu_{n}\}$ of $\C^{n+1}$ such that, for every $t\in\C^{*}$,
\[
\lambda(t)\nu_{i}=t^{r_{i}}\nu_{i}, \qquad r_{i}\in \Z.
\]
The number $r_{i}$ is called the {\bf weight} of $\nu_{i}$ with respect to the action of $\lambda$ on $\C^{n+1}$.

\begin{defn}
Let $x\in V\subset \P^{n}$ be a point, and let $\lambda$ be a $1$-PS of $G$. If $\overline{x}\in\C^{n+1}$ is a representative of the class $x$, and
\[
\overline{x}=\sum_{i=0}^{n}a_{i}\nu_{i},
\]
then
\[
\lambda(t)\overline{x}=\sum_{i=0}^{n} t^{r_{i}}a_{i}\nu_{i}.
\]
The {\bf Mumford function} is defined by
\[
\mu(x,\lambda):=\min\{r_{i}:a_{i}\neq 0\}.
\]
\end{defn}

The Hilbert--Mumford criterion, which is used to detect unstable points under a linear action, can be stated in the following numerical form.

\begin{thm}[Definition {\cite[Theorem 4.9]{newstead51lectures}}] \label{definicionEstabilidad}\hspace{0.3cm}
\begin{enumerate}
    \item A point $x$ is stable (respectively, semistable) if and only if $\mu(x,\lambda)<0$ (respectively, $\mu(x,\lambda)\leq 0$) for every $1$-PS $\lambda$ of $G$. The set of stable (respectively, semistable) points is denoted by $V^{s}$ (respectively, $V^{ss}$).
    
    \item A point $x$ is unstable if and only if there exists a $1$-PS $\lambda$ of $G$ such that $\mu(x,\lambda)>0$. The set of unstable points is denoted by $V^{un}$.
\end{enumerate}
\end{thm}

If we consider the open subset of semistable points $V^{ss}$ and restrict the action of $G$ to $V^{ss}$, then there exists a good quotient for this action. For example, in the case of foliations, Alcántara \cite{AlcantaraGrado1} showed that the good quotient for foliations of degree $1$ is
\[
\F(1,2)//SL_{3}(\C)=\P^{1}.
\]
Moreover, Kirwan showed that the unstable points contain information about the quotient variety \cite{Kirwan}, which is one of the reasons why their study is important.

Since $\F(d;2)$ is a projective variety of dimension $d^{2}+4d+2$, we may consider the action of $SL_{3}(\C)$ on $\F(d;2)$ induced by coordinate changes:
\[
\begin{array}{rcl}
  SL_{3}(\C)\times \F(d;2)&\rightarrow & \F(d;2) \\
    (g, \chi)&\mapsto & g\cdot \chi=Dg\,\chi(g^{-1}).
\end{array}
\]

It is well known that every one-parameter subgroup of $SL_{3}(\C)$ can be written as
\[
g\lambda(t)g^{-1}
=
g
\begin{pmatrix}
t^{k_{1}}&0&0\\
0&t^{k_{2}}&0\\
0&0&t^{k_{3}}
\end{pmatrix}
g^{-1},
\]
for some $g\in SL_{3}(\C)$, where $k_{1}\geq k_{2}\geq k_{3}$ and $k_{1}+k_{2}+k_{3}=0$.

We denote by $\lambda_{(k_{1},k_{2})}$ the diagonal $1$-PS
\[
\begin{pmatrix}
t^{k_{1}}&0&0\\
0&t^{k_{2}}&0\\
0&0&t^{k_{3}}
\end{pmatrix},
\]
and we assume that the integers $k_{1},k_{2},k_{3}$ are relatively prime.

Moreover, the action by coordinate changes interacts well with the invariants of interest: the Milnor number and the algebraic multiplicity of a singularity are invariant under this linear action.

Consider the following set of generators of $\F(d;2)$:
\[
\left\{ \mathcal{X}_{i_{0},j_{0}}^{0,d},\mathcal{ X}_{i_{1},j_{1}}^{1,d},\mathcal{X}_{i_{2},j_{2}}^{2,d}~\middle|~ j_{l}\geq i_{l}\ \forall\, l=0,1,2,\ \text{and } j_{l},i_{l}\in \{0,\ldots,d\} \right\},
\]
where
\[
    \mathcal{X}_{i_{0},j_{0}}^{0,d}= x^{d-j_{0}}y^{j_{0}-i_{0}}z^{i_{0}}\frac{\partial}{\partial x},\qquad
    \mathcal{X}_{i_{1},j_{1}}^{1,d}= x^{d-j_{1}}y^{j_{1}-i_{1}}z^{i_{1}}\frac{\partial}{\partial y},\qquad
    \mathcal{X}_{i_{2},j_{2}}^{2,d}= x^{d-j_{2}}y^{j_{2}-i_{2}}z^{i_{2}}\frac{\partial}{\partial z}.
\]

We may then consider a decomposition of the space into one-dimensional scalar subrepresentations of $\C^{d^2+4d+3}$ as
\[
\begin{split}
\C^{d^2+4d+3} =
&\bigoplus_{0\leq i_{0}\leq j_{0}\leq d}
a^{1-d-j_{0}}b^{i_{0}-j_{0}}c^{-i_{0}}\langle \mathcal{X}_{i_{0},j_{0}}^{0,d}\rangle \\
&\bigoplus_{0\leq i_{1}\leq j_{1}\leq d}
a^{j_{1}-d}b^{j_{1}-i_{1}+1}c^{-i_{1}}\langle \mathcal{X}_{i_{1},j_{1}}^{1,d}\rangle \\
&\bigoplus_{0\leq i_{2}\leq j_{2}\leq d}
a^{j_{2}-d}b^{i_{2}-j_{2}}c\langle \mathcal{X}_{i_{2},j_{2}}^{2,d}\rangle.
\end{split}
\]

We can associate with this representation a set of points in $\R^2$, called the \textit{weight diagram}, given by
\[
\left\{
-L_{1}(d-j_{l}-i_{l})-L_{2}(j_{l}-2i_{l})+L_{l+1}
\;\middle|\;
l=0,1,2
\right\}\subset \R^2,
\]
where
\[
L_1= (1,0), \qquad
L_2=\left(-\frac{1}{2},\frac{\sqrt{3}}{2}\right), \qquad
L_3=\left(-\frac{1}{2},-\frac{\sqrt{3}}{2}\right).
\]

It is well known that the numerical criterion in Theorem \ref{definicionEstabilidad} can be interpreted in terms of convexity properties of the weight diagram of the representation. More precisely, let $v=[v_0:\dots:v_n]\in V\subset \P^n$. Then:
\begin{itemize}
    \item The point $v$ is semistable if $0\in Conv\{\alpha_i\mid v_i\neq 0\}$.
    \item The point $v$ is stable if $0$ lies in the interior of $Conv\{\alpha_i\mid v_i\neq 0\}$.
    \item The point $v$ is unstable if $0\notin Conv\{\alpha_i\mid v_i\neq 0\}$.
\end{itemize}

Following results of C. Alcántara \cite{A10U}; see also \cite{R26}; a degree $d$ foliation with a singular point of algebraic multiplicity $d$ is unstable. In particular,
\[
\mathcal{A}_{d,d}\subset \F(d;2)^{un}.
\]

Intermediate cases are more difficult to analyze. As mentioned in the Introduction, it follows from \cite{cerveaudeserti} that every projective foliation is birational to a foliation having a unique singular point. A convenient birational transformation can be constructed using Cremona transformations. Nevertheless, as the following examples show, stability is not invariant under this type of transformation.

\begin{ex} \label{ex:Cremona}
    If we consider the foliation $\omega=Q_{d}dz+R_{d}(ydz-zdy)$ with $z\nmid Q_{d}$, without loss of generality, we can assume that $C(y^{d},Q_{d})=-1$. Then, $S(y,z)=zQ_{d}$. Assuming $z\nmid R_{d}$, by Lemma \ref{Lemma:Mk-noFactor},  $Sing(\omega)=\{p=[1:0:0],q=[a:1:0]\}$ where $a=C(y^{d},R_{d})\neq 0$. Moreover, by Lemma \ref{Lemma:Mk-noFactor}, $\mu_{p}(\omega)=d^{2}+d$ and $\mu_{q}(\omega)=1$.

    Consider the Cremona transformation given by \[\rho[x:y:z]=[xz:x^{2}-yz:z^{2}].\] This transformation has as indeterminacy set $Ind(\rho)=\{[0:1:0]\}$ and it has as exceptional set the curve $z=0$, $Exc(\rho)=\{ (z=0)\}$, thus $\rho(p)=\rho(q)=[0:1:0]=: q'$.
Now, 
  \[\rho^{\ast}\omega =(-z^{3}\tilde{Q}_{d}-2xz^{2}\tilde{R}_{d})dx+z^{3}\tilde{R}_{d}dy+(xz^{2}\tilde{Q}_{d}-yz^{2}\tilde{R}_{d}+2x^{2}z\tilde{R}_{d})dz,\]
   where $\tilde{Q}_{d}=Q_{d}(x^{2}-yz,z^{2})$ and $\tilde{R}_{d}=R_{d}(x^{2}-yz,z^{2})$. Since $z\nmid \tilde{R}_{d}$, we consider the foliation   defined by
\[\tilde{\omega}= (-z^{2}\tilde{Q}_{d}-2xz\tilde{R}_{d})dx+z^{2}\tilde{R}_{d}dy+(xz\tilde{Q}_{d}-yz\tilde{R}_{d}+2x^{2}\tilde{R}_{d})dz,\]
having $q'$ as only singularity. Since $deg(\tilde{R}_{d})=2d$, and  $deg(\tilde{Q}_{d})=2d$, it is a foliation of degree $2d+1$. 

Locally, in the open chart $y=1$ of the complex projective plane, the foliation  is given by
\begin{align*}
\tilde{\omega}|_{(y=1)}=&(-z^{2}\tilde{Q}_{d}(x,1,z)-2xz\tilde{R}_{d}(x,1,z))dx+\\ &(xz\tilde{Q}_{d}(x,1,z)-z\tilde{R}_{d}(x,1,z)+2x^{2}\tilde{R}_{d}(x,1,z))dz,\end{align*}
and the origin has multiplicity $d+1$, as $z^{d+1} dz$ appears explicitly in above expression.

So, denoting $s=2d+1$, we have constructed a family of  elements of  $\mathcal{A}_{s,\frac{s+1}{2}}$. 
\end{ex}

Following Example \ref{ex:Cremona}, consider the case $s=3$.  We can describe  part of the set $\mathcal{A}_{3,2}$, in addition to $\mathcal{A}_{3,3}$.  In the \cite{castorena2024git,pantaleon16problema}, the authors give explicit examples  in which a foliation of degree $3$ with a unique singular point of algebraic multiplicity $2$ can be stable, semistable not-stable, or unstable. In this example, if we consider a foliation of degree $1$ which has two singular points, without loss generality,  this foliation is the form $\omega_{1}= Q_{1}dz+R_{1}(ydz-zdy)$, where
$Q_{1}=-y$ and $R_{1}=ay+r_{2}z$, with  $r_{2}\in\C$ and $a\in\C^{*}$. Note that $\omega_{1}$ is given by the vector field $\XX_{1}=\begin{pmatrix}
    0 & -a & r_{0} \\ 0 & -1 & 0 \\ 0 & 0& 0
\end{pmatrix}\begin{pmatrix}
    x\\y\\z
\end{pmatrix}$. Since the matrix $\begin{pmatrix}
    0 & -a & r_{0} \\ 0 & -1 & 0 \\ 0 & 0& 0
\end{pmatrix}$ has two different eigenvalues,  the foliation $\omega_{1}$ is semistable  (see \cite{AlcantaraGrado1}).  After applying Cremona transformation $\rho$, we have the $1$-form \begin{align*}
   \tilde{\omega}=(-z^{2}\tilde{Q}_{1}-2xz\tilde{R}_{1})dx+z^{2}\tilde{R}_{1}dy+(xz\tilde{Q}_{1}-yz\tilde{R}_{1}+2x^{2}\tilde{R}_{1})dz,
\end{align*}
where $\tilde{Q}_{1}=-x^{2}+yz$, and $\tilde{R}_{1}=ax^{2}-ayz+r_{2}z^{2}$.

If  we consider a homogeneous vector field associated by the foliation of $1$-form  $\tilde{\omega}$ given by 

\begin{align*}
    \mathcal{X}=z\tilde{R}_{1} \frac{\partial}{\partial x}+ (z\tilde{Q}_{1}+2x\tilde{R}_{1})\frac{\partial}{\partial y}=\begin{pmatrix}
        ax^{2}z-ayz^{2}+r_{2}z^{3}\\
        2ax^{3}-x^{2}z-2axyz+2r_{2}xz^{2}+yz^{2}\\
        0
    \end{pmatrix}.
\end{align*}

We can see that $\mathcal{X}$ defines a unstable foliation, see for example the weight diagram of foliation of degree $3$ in \cite{pantaleon16problema}, and the curve $\{x^{2}-yz=0\}$ is invariant.

\section{Acknowledgments}
The second author gratefully acknowledges the hospitality and support of Universidad de Valladolid during a research stay, where this work was completed.

 \bibliographystyle{plain}
 \bibliography{Bibliografia}
        
\end{document}